\documentclass[12pt, reqno]{amsart}
\usepackage{palatino}
\usepackage{comment}
\usepackage{mathtools,amsmath,amssymb,amsxtra,color,calligra,mathrsfs,tcolorbox}
\usepackage[all,cmtip]{xy}
\usepackage{xcolor}
\colorlet{mdtRed}{red!50!black}
\definecolor{dblue}{rgb}{0,0,.6}
\usepackage[colorlinks]{hyperref}
\hypersetup{linkcolor=mdtRed,citecolor=dblue,filecolor=dullmagenta,urlcolor=mdtRed}
\usepackage[all]{xy}
\usepackage{tikz,tikz-cd,tkz-graph,enumerate}

\usetikzlibrary{matrix,arrows,decorations.pathmorphing}

\DeclareMathOperator{\Malpha}{\mathcal{M}_{\alpha}} 
\DeclareMathOperator{\Mbeta}{\mathcal{M}_{\beta}}
\DeclareMathOperator{\Mgamma}{\mathcal{M}_{\gamma}}
\DeclareMathOperator{\phialpha}{\phi_{\alpha}}
\DeclareMathOperator{\phibeta}{\phi_{\beta}}

\DeclareMathOperator{\Pnalpha}{\mathbb{P}^{n_{\alpha}}} 
\DeclareMathOperator{\Pnbeta}{\mathbb{P}^{n_{\beta}}}
\DeclareMathOperator{\Q}{\mathbb{Q}}
\DeclareMathOperator{\CH}{\textnormal{CH}}
\DeclareMathOperator{\C}{\mathbb{C}}
\DeclareMathOperator{\M}{\mathcal{M}}
\DeclareMathOperator{\N}{\mathcal{N}}

\newtheorem{theorem}{Theorem}[section]
\newtheorem{lemma}[theorem]{Lemma} 
\newtheorem{proposition}[theorem]{Proposition}
\newtheorem{corollary}[theorem]{Corollary}

\theoremstyle{definition}
\newtheorem{definition}[theorem]{Definition}

\numberwithin{equation}{section} 

\begin{document}
	
	\baselineskip=15.5pt 
	
	\title[On Abel-Jacobi maps of moduli of parabolic bundles]{On Abel-Jacobi maps of moduli of parabolic bundles over a curve}
	
	\author{Sujoy Chakraborty} 
	\address{School of Mathematics, Tata Institute of Fundamental Research, Homi Bhabha Road, Colaba, Mumbai 400005, India.}
	\email{sujoy@math.tifr.res.in}
	\thanks{E-mail : sujoy@math.tifr.res.in}
	\thanks{Address : School of Mathematics, Tata Institute of Fundamental Research, Homi Bhabha Road, Colaba, Mumbai 400005, India.}
	
	\begin{abstract}
		Let $C$ be a nonsingular complex projective curve, and $\mathcal{L}$ e a line bundle of degree 1 on $C$. Let $\Malpha := \mathcal{M}(r,\mathcal{L},\alpha)$ denote the moduli space of $S$-equivalence classes of Parabolic stable bundles of fixed rank $r$, determinant $\mathcal{L}$, full flags and generic weight $\alpha$. Let $n=$ dim$\Malpha$. We aim to study the Abel-Jacobi maps for $\Malpha$ in the cases $k=2,n-1$. When $k=n-1$, we prove that the Abel-Jacobi map is a split surjection. When $k=2$ and $r=2$, we show that the Abel-Jacobi map is an isomorphism.
	\end{abstract}
	
	\maketitle
	
	\section{Introduction}
	Let $X$ be a smooth projective variety over $\mathbb{C}$. The Abel-Jacobi map was defined by Griffiths as a generalization of the Jacobi map for curves. For each integer $k$, it is a map $AJ^k$ from the group $Z^k(X)_{hom}$ of codimension-$k$ cycles modulo the cycles homologous to zero, to the intermediate Jacobian $IJ^k(X) := \dfrac{H^{2k-1}(X,\C)}{F^kH^{2k-1}+ H^{2k-1}(X,\mathbb{Z})}$ (see Section 2 for the definitions), and the map $AJ^k$ actually factors through the group $\CH^k(X)_{hom} := \dfrac{Z^k(X)_{hom}}{Z^k(X)_{rat}}$, where the denominator is the subgroup of codimension-$k$ cycles rationally equivalent to zero (see e.g. \cite[Chapter 12]{Voi1}). In general $IJ^k(X)$ is a complex torus, and it is in fact an Abelian variety for $k=2,n-1$, where $n=dim(X)$. An interesting question is to study the \textit{weak representability} of the Abel-Jacobi map $AJ^k$ for $k=2, n-1$ and also determine the Abelian variety in terms of the cycles on $X$. In \cite{JY}, the authors have studied, among other things, these two Abel-Jacobi maps when $X$ is the moduli space of semistable vector bundles of fixed rank and determinant on a curve. Our aim here is to study these two Abel-Jacobi maps in the following situations:\\	
	Let $C$ be a nonsingular projective curve of genus $g\geq 3$ over $\mathbb{C}$. Let $\mathcal{L}$ be a line bundle on $C$. Let us moreover fix a finite set of closed points $S$ on $C$, referred to as \textit{Parabolic points}, and Parabolic weights. We also assume that the weights are generic. Let $\mathcal{M}(r,\mathcal{L}, \alpha)$ denote the moduli space of $S$-equivalence classes of Parabolic stable vector bundles of rank $r$, determinant $\mathcal{L}$, full flags at each Parabolic point, and \textit{generic} weights $\alpha:= \{0\leq \alpha_{1,P} <\alpha_{2,P}<\cdots <\alpha_{r,P}< 1\}_{P\in S} $ along the Parabolic points. \\
	For genric weights $\alpha$, $\mathcal{M}(r,\mathcal{L}, \alpha)$ is a smooth projective variety over $\mathbb{C}$ of dimension $n:= (r^2-1)(g-1)+\dfrac{m(m-1)}{2}$, where $m=|S|$. Our aim is to study some specific Abel-Jacobi maps when the variety in consideration is the moduli space $\mathcal{M}(r,\mathcal{L}, \alpha)$. We now outline the content of the paper and the main ideas of the proofs below.\\
	In Section 2, we recall all the notions and definitions needed for our purpose, for example the notion of (Parabolic) stability and semi-stability of (Parabolic) vector bundles on a curve, their moduli spaces and so on.\\
	In Section 3, we fix a line bundle $\mathcal{L}$ of degree 1, and study the Abel-Jacobi map $AJ^{n-1}$ for the space $\Malpha := \mathcal{M}(r,\mathcal{L}, \alpha)$. Let us denote by $\mathcal{M} := \mathcal{M}(r,\mathcal{L})$ the moduli space of isomorphism classes of stable vector bundles of rank $r$ and determinant $\mathcal{L}$ on $C$. The main result of Section 3 is the following:
	
	\begin{theorem}[Theorem \ref{abeljacobiarbitwt}]
		Let $n=$ dim$\Malpha$. For any generic weight $\alpha$,
		\[AJ^{n-1}: \CH^{n-1}(\Malpha)_{hom}\otimes \Q \rightarrow IJ^{n-1}(\Malpha)\otimes\Q\] is a split surjection. 
	\end{theorem}
	We first prove the above in the case when the weights $\alpha$ are sufficiently small, in the sense of \cite[Proposition 5.2]{BY}. In this case, there exists a morphism $\pi: \Malpha \rightarrow \mathcal{M}$ by forgetting the Parabolic structure, under which $\Malpha$ actually becomes a Zariski locally trivial fibration with fibres a product of flag varieties by \cite[Theorem 4.2]{BY}. Our approach is based on \cite{JY}, where the result has been proved for $\mathcal{M}$. Regarding the choice of determinant, a crucial input in the proof is \cite[Theorem 2.1]{JY}, where the determinant $\mathcal{O}(x)$ is necessary. Next, to prove the result for arbitrary generic weights, we use \cite[Theorem 4.1]{BY}, which says that if $\alpha,\beta$ are generic weights separated by a single wall in the set of all compatible weights (see Section 2 for definitions), and if $\gamma$ is the weight which is the point of intersection of the wall and the line joining $\alpha$ and $\beta$, then there exist maps
	\[
	\xymatrix{
		\Malpha \ar[rd]_{\phi_{\alpha}}
		& 
		&\Mbeta \ar[ld]^{\phi_{\beta}} \\
		&\Mgamma
	}
	\] 
	which act as resolution of singularities for $\Mgamma$. The fibre product $\N := \Malpha \underset{\Mgamma}{\times}\Mbeta$ turns out to be a common blow-up of $\Malpha,\Mbeta$ along suitable subvarieties, which helps us to relate the Abel-Jacobi maps for $\Malpha$ and $\Mbeta$ via that of $\N$. This, together with the fact that the result was already shown for $\alpha$ sufficiently small, enables us to conclude Theorem 1.1 for any generic weight.\\
	In Section 4, we fix rank $r=2$ and determinant a degree 1 line bundle $\mathcal{L}$. The main result of this section is the following:
	\begin{theorem}[Theorem \ref{abeljacobiiso}]
		For any generic weight $\alpha$, $AJ^2: \CH^2(\Malpha)_{hom}\otimes \Q \rightarrow IJ^{2}(\Malpha)\otimes\Q$ is an isomorphism. 	
	\end{theorem}
	Again, our approach is to first show the result for sufficiently small generic weights, and then show it for arbitrary weights. Let $\Malpha, \Mbeta, \Mgamma$ be as in above. As for the reason to fix the rank to be 2, the main step in proving the result for arbitrary weights is to find an explicit description of the singular locus $\Sigma_{\gamma}$ of $\Mgamma$, which we describe in Proposition \ref{sigmagammadescription}, which actually uses that we are in the case $r=2$. \\
	We would also like to point out that even though we study the Abel-Jacobi map with $\Q$-coefficients, most of the arguments will hold for $\mathbb{Z}$ coefficients as well, except some of the results (e.g. Proposition \ref{chow1iso}(b)), where we use dimension reasoning.	
	
	\section{Preliminaries}
	
	\subsection{Semistability and stability of vector bundles}
	
	Let $C$ be a nonsingular projective curve over $\mathbb{C}$. Let $E$ be a holomorphic vector bundle of rank $r$ over $C$. \\
	Here onwards, by a $variety$ we will always mean an irreducible quasi-projective variety.
	
	\begin{definition}[Degree and slope]
		The \textit{degree} of $E$, denoted $deg(E)$, is defined as the degree of the line bundle $det(E) := \wedge^r E$. The \textit{slope} of $E$, denoted $\mu(E)$, is defined as 
		\[\mu(E) := \dfrac{deg(E)}{r}\]
	\end{definition}
	
	\begin{definition}[Semistability and stability]
		
		$E$ is called \textit{semistable (resp. stable)}, if for any sub-bundle $F\subset E, \, 0< rank(F) < r$, we have 
		\[\mu(F) \, \underset{(resp. <)}{\leq}\,\mu(E).\]
		
	\end{definition}
	
	It is easy to check that if $gcd(r, deg(E))=1$, then the notion of semistability and stability coincide for a vector bundle $E$.
	
	\subsection{Moduli space of vector bundles}\label{sec-2.2}
	We briefly recall the notion of the moduli space of vector bundles over $C$. The collection of all semistable vector bundles over $C$ of fixed slope $\mu$ forms an abelian category, whose simple objects are exactly the stable vector bundles of slope $\mu$. If $E$ is a semistable bundle of slope $\mu$, then there exists a \textit{Jordan-H\"older filtration} for $E$ given by 
	\[E = E_k \supset E_{k-1} \supset \cdots \supset E_1 \supset 0\]
	
	such that each $E_i/E_{i-1}$ is stable bundle of same slope $\mu$. The filtration is not unique, but the associated graded object $\textrm{gr}(E) := \bigoplus_{i=1}^{k} E_i/E_{i-1}$ is unique upto isomorphism. Two vector bundles $E$ and $E'$ are called \textit{S-equivalent} if $\textrm{gr}(E)\cong \textrm{gr}(E')$. When $E, E'$ are stable, being S-equivalent is same as being isomorphic as vector bundles over $C$.
	
	The moduli space of S-equivalence classes of vector bundles of rank $r$ and determinant $\mathcal{L}$ on $X$, denoted $\mathcal{M}(r,\mathcal{L})$, is a normal projective variety of dimension $(r^2-1)(g-1)$; its singular locus is given by the strictly semistable bundles. 
	
	In the case when $gcd(r,deg(\mathcal{L)})=1$, $\mathcal{M}(r,\mathcal{L})$ is the isomorphism class of stable vector bundles on $C$. It is a nonsingular projective variety; moreover, it is a fine moduli space.
	
	When $r,\mathcal{L}$ are fixed, we shall denote the moduli space by $ \mathcal{M} $, when there is no scope for confusion.
	\subsection{Parabolic bundles and stability}
	
	\begin{definition}[Parabolic bundles, Parabolic data]
		Let us fix a set $S = {p_1, \cdots , p_n}$ of $n$ distinct closed points on $C$. A \textit{Parabolic vector bundle of rank r on C} is a holomorphic vector bundle $E$ on $C$ with a \textit{Parabolic structure along points of S}. By this, we mean a collection weighted flags of the fibers of $E$ over each point $p\in S$:
		
		\begin{align}
		E_p &= E_{p,1} \supsetneq E_{p,2} \supsetneq ... \supsetneq E_{p,s_p} \supsetneq E_{p, s_{p+1}}= 0, \\
		0 &\leq \alpha_{p,1} < \alpha_{p,2} < ... \,< \alpha_{p,s_p}\, < 1,
		\end{align}
		
		where $s_p$ is an integer between $1$ and $r$. The real number $\alpha_{p,i}$ is called the \textit{weight attached to the subspace} $E_{p,i}$. 	
		The \textit{multiplicity} of the weight $\alpha_{p,i}$ is the integer $m_{p,i} := dim(E_{p,i}) - dim(E_{p,i-1})$. Thus $\sum_i m_{p,i} = r$. 
		We call the flag to be $full$ if $s_p=r,$ or equivalently $m_{p,i} =1 \,\forall i$.
		Let $\alpha := \{(\alpha_{p,1}, \alpha_{p,2}, ..., \,\alpha_{p,s_p}\,)\,|\,p\in S\}$ and $m:= \{(m_{p,1},...,\,m_{p,s_p}\,)\,|\, p\in S\}$. We call the tuple $(r,\,\mathcal{L},\,m,\,\alpha)$ as the \textit{Parabolic data} for the Parabolic bundle $E$, where $\mathcal{L} := det(E)$. Usually we denote the Parabolic bundle as $E_*$ to distinguish from the underlying vector bundle $E$.
	\end{definition}
	
	\begin{definition}[Parabolic degree and slope]
		The \textit{degree} of a Parabolic bundle $E_*$ is defined as the usual degree $deg(E)$ of the underlying vector bundle $E$. The \textit{Parabolic degree} of $E_*$ with respect to $\alpha$ is defined as
		\[Pardeg_{\alpha}(E_*):= deg(E) + \sum_{p\in S}\sum_{i=1}^{s_P}m_{p,i}\alpha_{p,i}.\]
		
		The \textit{Parabolic slope} of $E_*$ with respect to $\alpha$ is defined as
		\[Par\mu_{\alpha}(E_*) := \dfrac{Pardeg_{\alpha}(E_*)}{rank(E)}.\]
	\end{definition}
	
	\begin{definition}[Parabolic semistability and stability]\label{pardegslopedef}
		Any vector sub-bundle $F\hookrightarrow E$ obtains a Parabolic structure in a canonical way: For each $p\in S$, the flag at $F_p$ is obtained intersecting $F_p$ with the flag at $E_p$, and the weight attached to the subspace $F_{p,j}$ is $\alpha_k$, where $k$ is the largest integer such that $F_{p,j}\subseteq E_{p,k}$. (for more details see \cite[Definition 1.7]{MS}.) We call the resulting Parabolic bundle to be a \textit{Parabolic sub-bundle,} and denote it by $F_*$.
		
		A Parabolic bundle $E_*$ is called $\alpha$-\textit{Parabolic semistable (resp. $\alpha$-Parabolic stable)}, if for every proper sub-bundle $F\subset E, 0<rank(F)<rank(E)$, we have
		\[Par\mu_{\alpha}(F_*)\, \underset{(resp. <)}{\leq} \,Par\mu_{\alpha}(E_*).\]  
		We also call them simply Parabolic semistable or Parabolic stable, if the weight is clear from the context.
	\end{definition}
	
	\subsection{Generic weights and walls}
	
	We briefly recall the notion of \textit{generic weights} and \textit{walls}. For more details we refer to \cite{BH}.
	
	Fix a set $S$ of points in $C$, positive integer $r$, line bundle $\mathcal{L}$ on $C$ and multiplicities $m$ as defined above. Let $\Delta^r:= \{(a_1,..., a_r)\, | \, 0\leq a_1 \leq ... \leq a_r <1\}$, and define $W := \{\alpha : S \rightarrow \Delta^r\}$. Note that the elements of $W$ determine both weights and the multiplicities at the Parabolic points, and hence a Parabolic structure. Conversely, given multiplicities $m$ at the Parabolic points, we can associate a map $S \rightarrow \Delta^r$, by repeating each weight $\alpha_{p,i}$ according to its multiplicty $m_{p,i}$.
	This leads to a natural notion of when a given weight $\alpha$ is \textit{compatible} with the multiplicity $m$. 
	The set of all weights compatible with $m$ is a product of $|S|$-many simplices. We denote by $V_m$ the set of all weights compatible with m.
	
	Let $\alpha \in V_m$. Let $E_*$ be a Parabolic bundle with data $(r, \mathcal{L}, m, \alpha)$ and Parabolic degree 0. Let $d= deg(\mathcal{L})$. If $E_*$ is Parabolic semistable but not Parabolic stable, then it would contain a sub-bundle $F$ of rank $r'$ and degree $d'$ (say), such that under the induced Parabolic structure on $F$, with induced weights $\alpha' := \{0\leq \alpha'_{p,1} \leq \cdots \leq \alpha'_{p,r'}< 1\}_{p\in S}$ where each $\alpha'_{p,i}=\alpha_{p,j}$ for some $j$, we get $Pardeg_{\alpha'}(F_*)=Pardeg(E_*)$. This translates to
	\begin{align*}
	r(d'+\sum_{p\in S}\sum_{i=1}^{r'}\alpha'_{p,i}) = r'(d+\sum_{p\in S}\sum_{j=1}^{r}\alpha_{p,j}).
	\end{align*}
	This clearly gives a hyperplane section of $V_m$. Such a hyperplane is determined by the data $\xi := (r',d',\alpha')$. Hence, it is easy to see that only finitely many such hyperplanes can intersect $V_m$ (see \cite{BH}); call them $H_1,\cdots,H_l$.
	
	\begin{definition}(Walls and generic weights)
		We call each of the intersections $H_i \cap V_m$ a \textit{wall} in $V_m$. There are only finitely many such walls.
		
		We call the connected components of $V_m \setminus \cup_{i=1}^{l}H_i$ as \textit{chambers}, and weights belonging to these chambers are called \textit{generic}.
		
	\end{definition}
	
	Clearly, for generic weights, a Parabolic bundle is Parabolic semistable iff it is Parabolic stable.
	
	\subsection{Moduli of Parabolic bundles}\label{sub-2.5}
	Again, we briefly recall the notion of moduli space of Parabolic semistable bundles over $X$. The construction is analogous to section \ref{sec-2.2}; for details we refer to \cite{MS}.\\
	
	\begin{definition}[\text{\cite[Definition 1.5]{MS}}]\label{defparhomo}
		A \textit{morphism of Parabolic bundles} $f: E_*\rightarrow E'_*$ is a morphism of usual bundles $E\rightarrow E'$, such that at each Parabolic point $P$, denoting by $f_P$ the restriction of $f$ to the fibre at $P$, we have $f_P(E_{P,i})\subseteq E'_{P,j}$ whenever $\alpha_{P,i}>\alpha'_{P,j}$.
	\end{definition}
	
	The collection of Parabolic semistable bundles $E_*$ with fixed Parabolic slope form an abelian category. For each such $E_*$ there exists a Jordan-H\"{o}lder filtration similar to Section \ref{sec-2.2}, with the obvious difference that each successive quotient of the filtration is Parabolic stable and of the same Parabolic slope as that of $E_*$. and we can define an associated graded object $gr_{\alpha}(E_*)$ analogous to section \ref{sec-2.2}. Again, we call two Parabolic semistable bundles to be $S$-equivalent if their associated graded objects are isomorphic. 
	
	\begin{definition}
		We denote by $\mathcal{M}(r,\mathcal{L},m,\alpha)$ the moduli space of S-equivalence classes of Parabolic semistable bundles over $C$ of rank $r$, determinant $\mathcal{L}$, multiplicities $m$ and weights $\alpha$. It is a normal projective variety, with singular locus given by the strictly semistable bundles. When $r,\mathcal{L},m$ are fixed, we will denote the moduli space by $\Malpha$ if no confusion occurs.\\
		For generic weight $\alpha$, $\mathcal{M}_\alpha $ is a nonsingular projective variety; moreover, it is a fine moduli space (\cite[Proposition 3.2]{BY}). 
	\end{definition}
	
	\subsection{Chow groups}
	Let $X$ be a variety over $\mathbb{C}$ of dimension $n$.  let $Z^k(X)$ (equivalently, $Z_{n-k}(X)$) denote the free abelian group generated by the irreducible closed subvarieties of codimension $k$ (equivalently, dimension $n-k$) in $X$. There are many interesting equivalence relations on this group; we require two of them for our purpose:
	namely, \textit{rational equivalence} and \textit{homological equivalence}. We refer to \cite[Section 9]{Voi2} and \cite{Ful} for further details. Let $Z^(X)_{hom}$ and $Z^k(X)_{rat}$ denote the subgroups of $Z^k(X)$ consisting of elements homologically equivalent to zero and rationally equivalent to 0, respectively. In general the following chain of containments hold: $Z^k(X)_{rat}\subseteq Z^k(X)_{hom}\subseteq Z^k(X)$. We define 
	\[\CH^k(X) := \dfrac{Z^k(X)}{Z^k(X)_{rat}},\, \CH^k(X)_{hom} := \dfrac{Z^k(X)_{hom}}{Z^k(X)_{rat}}.\]	
	Clearly $\CH^k(X)_{hom}\subset \CH^k(X)$. We also denote them by $\CH_{n-k}(X)_{hom}$ and $\CH_{n-k}(X)$ respectively, if we want to focus on the dimension rather than codimension of the subvarieties. 
	\subsection{Intermediate Jacobians}
	Let $X$ be a smooth projective variety over $\mathbb{C}$. For each $k\geq 0$, we have the Hodge decomposition
	\[H^k(X,\mathbb{C})\cong \bigoplus_{i+j=k, p,q\geq 0} H^{i,j}(X),\]
	where $H^{i,j}(X)$ are the Dolbeault cohomology groups, with the property that $\overline{H^{i,j}(X)} = H^{j,i}(X) \,\,\forall i,j$.
	Let $k=2p-1$ be odd, and define $F^pH^{2p-1}(X,\mathbb{C}) := H^{2p-1,0}(X)\oplus H^{2p-2,1}(X)\oplus\cdots \oplus H^{p,p-1}(X)$. Then we have \[H^{2p-1}(X,\C) = F^pH^{2p-1} \oplus \overline{F^pH^{2p-1}}\] 
	Moreover, the image of the composition 
	\[H^{2p-1}(X,\mathbb{Z})\rightarrow H^{2p-1}(X,\C) \twoheadrightarrow \overline{F^pH^{2p-1}}\] gives a lattice.	
	\begin{definition}\label{defintjac}
		We define the \textit{p-th Intermediate Jacobian} as 
		\[IJ^p(X) := \dfrac{H^{2p-1}(X,\C)}{F^pH^{2p-1}(X,\mathbb{C})+ H^{2p-1}(X,\mathbb{Z})} = \dfrac{\overline{F^pH^{2p-1}(X,\mathbb{C})}}{H^{2p-1}(X,\mathbb{Z})}.\]
	\end{definition}
	
	\section{Surjectivity of Abel-Jacobi map for 1-cycles}	
	
	Let us fix an integer $r$, a line bundle $\mathcal{L}$ of degree 1, and let $\M$  denote the moduli space of semistable bundles of rank $r$ and determinant $\mathcal{L}$, and similarly let the moduli of parobolic bundles  be denoted by $\Malpha := \M(r,\mathcal{L},m,\alpha)$.	
	\subsection{CASE OF SMALL WEIGHTS}\label{smallweight}
	
	First, let us assume that the Parabolic weights are small enough, so that \cite[Proposition 5.2]{BY} is applicable. Such small generic weights exist by \cite[Proposition 3.2]{BY}, since the bundles have degree 1 and there are only finitely many walls. 
	
	In this case, by \cite[Proposition 5.2]{BY}, there exists a morphism 
	\[\pi : \Malpha\rightarrow \M,\]
	by forgetting the Parabolic structure; moreover, $\pi$ is a locally trivial fibration (in Zariski topology) with fibers as product of flag varieties by \cite[Theorem 4.2]{BY}.
	
	\begin{proposition}\label{specseq}
		$\pi^* : H^3(\M,\Q)\rightarrow H^3(\Malpha,\Q)$ is an isomorphism.
	\end{proposition}
	
	\begin{proof}
		(see \cite[Theorem 4.11]{Voi2} for the description of Leray spectral sequence.)
		Consider the Leray spectral sequence of $\pi$ with rational coefficients, satisfying
		\[E_{2}^{p,q} = H^p(\M, R^q \pi_* \underline{\Q}_{\Malpha}) \implies H^{p+q}(\M,\Q)\]
		where $\Q_{\Malpha}$ denotes the locally constant sheaf with stalk $\Q$.
		Since $\pi$ is a fibration, the sheaves $R^q \pi_* \underline{\Q}_{\Malpha}$ turn out to be local systems (locally constant sheaves with fibers as finite dimensional $\Q$-vector space); moreover, since by \cite[Theorem 1.2]{KS} $\M$ is rational, it is simply connected. If $F$ denotes the fiber of $\pi$, by the equivalence of categories of local systems and representations of $\pi_1(\M)$, we conclude that all the $R^q \pi_* \underline{\Q}_{\Malpha}$ are constant sheaves with stalk $H^q(F,\Q)$,
		\begin{eqnarray*}
			\therefore H^p(\M,R^q \pi_* \underline{\Q}_{\Malpha}) = H^p(\M, H^q(F,\Q)) \overset{UCT}{\simeq} H^p(\M, \Q)\otimes_{\Q} H^q(F,\Q).
		\end{eqnarray*} 
		By a theorem of Deligne (\cite[Theorem 4.15]{Voi2}), since $\pi$ is smooth and proper, the above spectral sequence degenerates at $E_2$-page, i.e. $E_2^{p,q} = E_{\infty}^{p,q} \,\,\forall p,q.$
		Now, $H^3(\Malpha,\Q)$ has a filtration 
		\begin{eqnarray}\label{filtraion}
		0\subseteq F^3H^3\subseteq F^2H^3\subseteq F^1H^3\subseteq F^0H^3 = H^3(\Malpha,\Q),
		\end{eqnarray}
		with $E_2^{p,3-p} = E_{\infty}^{p,3-p}\simeq \dfrac{F^pH^3}{F^{p+1}H^3}\,\,\forall p\leq3.$
		Let us compute the various $E_2^{p,3-p}$'s:
		\begin{eqnarray*}
			E_2^{3,0} &=& H^3(\M,\Q),\\
			E_2^{2,1} &=& H^2(\M, H^1(F,\Q)) = 0\,\,(\because \text{F has cohomologies only in even degress}),\\
			E_2^{1,2} &=& H^1(\M,H^2(F,\Q)) = 0\,\,(\because \M \text{is simply connected}),\\
			E_2^{0,3} &=& H^0(\M,H^3(F,\Q)) = 0 \,\,(\because \text{F has cohomologies only in even degress})
		\end{eqnarray*}
		
		$\therefore$ the filtration in \ref{filtraion} looks like 
		\[0\subseteq F^3H^3 = ... = F^0H^3=H^3(\Malpha,\Q) ,\]
		from which we conclude that $H^3(\M,\Q) = E_2^{3,0}=E_{\infty}^{3,0}\simeq F^3H^3 = F^0H^3 = H^3(\Malpha,\Q)$.
		Moreover, the isomorphism is precisely the edge map $E_\infty^{3,0}\rightarrow H^3(\Malpha,\Q)$, which coincides with $\pi^*$. Hence we get our result.
	\end{proof}
	
	Fix a universal (Poinc\'are) bundle $U \rightarrow C\times \M$, and let $p_1: C\times\M \rightarrow C, p_2: C\times \M \rightarrow \M$ be the projections. Let $c_2(U)\in H^4(C\times\M,\Q)$ denote the second Chern class of $U$. Consider the following homomorphisms:
	\begin{eqnarray}
	H^1(C,\Q)\overset{p_{1}^{*}}{\rightarrow} H^1(C\times \M,\Q) \overset{\cup c_2(U)}{\longrightarrow}H^5(C\times \M,\Q)\overset{p_{2*}}{\longrightarrow}H^3(\M,\Q) \label{3.2}
	\end{eqnarray}
	
	The last map, namely 'pushforward' $p_{2*}$ is defined by the composition
	\begin{equation*}
	\resizebox{1.0\hsize}{!}{$H^5(C\times \M,\Q)\xrightarrow[\simeq]{\overset{\text{Poincar\'e}}{\text{duality}}} H_{2n-5}(C\times\M,\Q) \xrightarrow[\simeq]{UCT} H^{2n-5}(C\times\M,\Q)\,\check{}\xrightarrow{(p_2^{*})\check{}}H^{2n-5}(\M,\Q)\check{}\xrightarrow[\simeq]{UCT}H_{2n-5}(\M,\Q) \xrightarrow[\simeq]{\overset{\text{Poincar\'e}}{\text{duality}}}H^3(\M,\Q)$}
	\end{equation*}	
	where UCT denotes the isomorphism given by the Universal coefficient theorem. \\
	The composition $\Gamma_{c_2(U)} := p_{2*}\circ \cup c_2(U) \circ p_1^*$ is called the \textit{correspondence} defined by $c_2(U)$ in literature.	
	\begin{proposition}\label{cohoiso}
		$\Gamma_{c_2(U)} : H^1(C,\Q) \rightarrow H^3(\M,\Q)$ is an isomorphism for $r\geq 2 , g\geq 3$.
	\end{proposition} 
	
	\begin{proof}
		Follows from \cite[Theorem 2.1]{JY}, by observing that the argument only uses the fact that the determinant is a line bundle of degree 1.
	\end{proof}	
	Let $U' := (Id\times\pi)^*(U)$ denote the pullback of $U$ onto $C\times\Malpha$, and define $\Gamma_{c_2(U')}$ in the same way as  $\Gamma_{c_2(U)}$ in \eqref{3.2} by replacing $\M$ by $\Malpha$ and $U$ by $U'$ .
	
	\begin{lemma}
		$\Gamma_{c_2(U')}: H^1(C,\Q) \rightarrow H^3(\Malpha,\Q)$ is isomorphism as well.
	\end{lemma}
	
	\begin{proof}
		Consider the following commutative diagrams:		
		\begin{align*}
		\xymatrix{H^1(C,\Q) \ar[r]^(.45){p_1^*} \ar[d]^{Id} & H^1(C\times\Malpha,\Q) \ar[r]^{\cup c_2(U')} \ar[d]^{(Id\times\pi)^*}
			& H^5(C\times\Malpha,\Q) \ar[r]^(.56){p_{2*}} \ar[d]^{(Id\times\pi)^*} 
			& H^3(\Malpha,\Q) \ar[d]^{\pi^*} \\
			H^1(C,\Q) \ar[r]^(.45){p_1^*} & H^1(C\times \M,\Q) \ar[r]^{\cup c_2(U)} 
			& H^5(C\times\M,\Q) \ar[r]^(.56){p_{2*}} & H^3(\M,\Q)
		}
		\end{align*} 
		The rightmost vertical map $\pi^*$ is an isomorphism by Proposition \ref{specseq}. The lower horizontal composition is equal to $\Gamma_{c_2(U)}$ by definition, which is also an isomorphism by Proposition \ref{cohoiso}. Hence we conclude that the upper horizontal composition, which equals $\Gamma_{c_2(U')}$, is an isomorphism as well.		
	\end{proof}
	
	Next, let $\mathcal{O}(1)$ be a very ample line bundle on $\Malpha$, and let $H:=c_1(\mathcal{O}(1))$ be its first Chern class. If $n=$ dim$\Malpha$, then there exists Hard Lefschetz isomorphism
	
	\begin{align}\label{hardlefschetz}
	H^3(\Malpha,\Q)\xrightarrow[\simeq]{\cup H^{n-3}} H^{2n-3}(\Malpha,\Q)
	\end{align}	
	Together with Lemma 3.3, we obtain an isomorphism
	\begin{align}\label{iso1}
	\cup H^{n-3}\circ\Gamma_{c_2(U')} : H^1(C,\Q)\xrightarrow{\sim} H^{2n-3}(\Malpha,\Q)
	\end{align}
	
	\begin{proposition}\label{iso2}
		The isomorphism \ref{iso1} induces an isomorphism 
		\begin{align}
		Jac(C)\otimes \Q \xrightarrow{\sim} IJ^{n-1}(\Malpha)\otimes \Q = \dfrac{H^{2n-3}(\Malpha,\C)}{F^{n-1}H^{2n-3}+H^{2n-3}(\Malpha,\Q)}
		\end{align}
		where $Jac(C)$ denote the isomorphism class of degree 0 line bundles on $C$.
	\end{proposition}
	
	\begin{proof}
		We recall the well-known fact that $Jac(C) \cong IJ^1(C) = \dfrac{H^1(C,\C)}{H^{1,0}(C)+H^1(C,\mathbb{Z})}$, which can be seen from the exponential exact sequence. The Hodge decomposition gives $H^{2n-3}(\Malpha,\C) = H^{n,n-3}\oplus \cdots\oplus H^{n-3,n}$; moreover, since the Hard Lefschetz isomorphism \ref{hardlefschetz} (after tensoring by $\C$) respects Hodge decomposition and the fact that $H$ is a type (1,1)-form (\cite[Proposition 11.27]{Voi1}), we have that under \ref{hardlefschetz}, $H^{3,0}(\Malpha)\cong H^{n,n-3}(\Malpha)$ and $H^{0,3}(\Malpha)\cong H^{n-3,n}(\Malpha)$. But since $\Malpha$ is rational variety, $H^{3,0}(\Malpha)$ $=H^{0,3}(\Malpha) = 0$; hence 
		\[H^{2n-3}(\Malpha) = H^{n-1,n-2}\oplus H^{n-2,n-1}.\]
		Since $p_1^*, \cup c_2(U'), p_{2*}, \cup H^{n-3}$ are morphisms of Hodge structures of weights $(0,0), (2,2),$\\
		$ (-1,-1)$ and $ (n-3,n-3)$ respectively, it is easy to see that the isomorphism \ref{iso1} takes $H^{1,0}(C)$ to $H^{n-1,n-2}(\Malpha)$ $= F^{n-1}H^{2n-3}(\Malpha,\mathbb{C})$. Hence we get our required isomorphism by going modulo appropriate subgroups.
	\end{proof}
	
	Next we want to show that the isomorphism \ref{iso2} is compatible with the similar 'correspondence map' on the rational Chow groups, which is defined analogous to $\Gamma_{c_2(U)}$: consider the composition
	\begin{align*}
	\CH^1(C)\otimes{\Q}\overset{p_{1}^{*}}{\rightarrow} \CH^1(C\times \Malpha)\otimes{\Q} \overset{\cap c_2(U')}{\longrightarrow}\CH^3(C\times \Malpha)\otimes{\Q}\overset{p_{2*}}{\longrightarrow}\CH^2(\Malpha)\otimes{\Q},
	\end{align*}
	where $U'$ is the pullback of $U$ as before, $c_2(U')\in \CH^2(C\times \Malpha)$ is the algebraic Chern class, $\cap$ denotes the intersection product and $p_1^*, p_{2*}$ denote the flat pullback and proper pushforward respectively.\\
	Define $\Gamma_{c_2(U')}^{\CH} := p_{2*} \circ \cap c_2(U')\circ p_1^*$. This restricts to the subgroups of cycles homologically equivalent to zero, giving 
	\[\Gamma_{c_2(U')}^{\CH}: \CH^1(C)_{hom}\otimes\Q \rightarrow \CH^2(\Malpha)_{hom}\otimes \Q\]
	
	Let $H := c_1(\mathcal{O}(1))$ denote the algebraic first chern class of a very ample line bundle on $\Malpha$.
	\begin{proposition}
		The following diagram commutes:
		\begin{align}
		\xymatrixcolsep{9pc}\xymatrix{\CH^1(C)_{hom}\otimes\Q \ar[r]^{\cap H^{n-3}\circ\Gamma_{c_2(U')}^{\CH}} \ar[d]^{AJ^1}_{\cong} 
			& \CH^{n-1}(\Malpha)_{hom}\otimes \Q \ar[d]^{AJ^{n-1}} \\
			Jac(C)\otimes \Q \ar[r]^{\cong} & IJ^{n-1}(\Malpha)\otimes \Q
		}
		\end{align}
		where the lower horizontal isomorphism is given by Proposition \ref{iso2}.	
		
	\end{proposition}
	
	\begin{proof}
		(\cite[Lemma 2.4]{JY}) By \cite[Page 44 (7.9)]{EV}, there exists a short exact sequence which relates the Intermediate Jacobian with the Deligne cohomology group:
		\[0\rightarrow IJ^{n-1}(\Malpha)\rightarrow H^{2n-2}_{\mathcal{D}}(\Malpha, \mathbb{Z}(n-1))\rightarrow Hg^{n-1}(\Malpha)\rightarrow 0\]
		Moreover, there exists a Deligne cycle class map
		\[\CH^{n-1}(\Malpha)\rightarrow H^{2n-2}_{\mathcal{D}}(\Malpha,\mathbb{Z}(n-1))\]
		which, when restricted to the subgroup $\CH^{n-1}(\Malpha)_{hom}$, factors through the subgroup $IJ^{n-1}(\Malpha)$ and coincides with the Abel-Jacobi map $AJ^{n-1}$. Furthermore, the Deligne cycle class map is compatible with the correspondence map and intersection product on Chow groups; in other words, the diagram in question commutes. The left vertical Abel-Jacobi map is an isomorphism by \cite[Theorem 11.1.3]{BL}, since in case of a smooth curve $\CH^1(C)_{hom} = Pic^0(C)$.
	\end{proof}
	
	\begin{corollary}\label{abeljacobismallwt}
		The Abel-Jacobi map $AJ^{n-1}: \CH^{n-1}(\Malpha)_{hom}\otimes \Q \rightarrow IJ^{n-1}(\Malpha)\otimes\Q$ is a split surjection.
	\end{corollary}
	
	\begin{proof}
		Follows from above lemma.
	\end{proof}
	
	\subsection{CASE OF ARBITRARY WEIGHTS}\label{arbitraryweight}
	
	Let $\alpha, \beta$ be two generic weights lying in adjacent chambers in $V_m$ separated by a single wall. Let $H$ be the hyperplane separating them. Let $\gamma$ be the intersection of $H$ and the line joining $\alpha$ and $\beta$. Then $\Malpha$ and $\Mbeta$ are smooth projective varieties, while $\Mgamma$ is normal projective, with the singular locus $\Sigma_{\gamma}$ given by the strictly semistable bundles.\\
	By \cite[Theorem 3.1]{BH}, there exist projective morphisms
	
	\[
	\xymatrix{
		\Malpha \ar[rd]_{\phi_{\alpha}}
		& 
		&\Mbeta \ar[ld]^{\phi_{\beta}} \\
		&\Mgamma
	}
	\] 
	
	such that (i) $ \phialpha $ and $ \phibeta $ are isomorphisms along $ \Mgamma \setminus \Sigma_{\gamma} $, \\
	\hspace*{10.7ex}(ii) $\phi_{\alpha}^{-1}(\Sigma_{\gamma})$ (resp. $\phi_{\beta}^{-1}(\Sigma_{\gamma})$) is a $\Pnalpha$-bundle (resp. $\Pnbeta$-bundle) over $\Sigma_{\gamma}$, \\
	\hspace*{5.5ex} and (iii) $n_{\alpha}+n_{\beta}+1 = codim \Sigma_{\gamma}$.
	
	Let $\N := \Malpha \underset{\Mgamma}{\times}\Mbeta$. It follows from the discussion at the end of section 1 in \cite{BH}that $\N$ is the common blowup over $\Malpha$ and $\Mbeta$ along $\phi_{\alpha}^{-1}(\Sigma_{\gamma})$ and $\phi_{\beta}^{-1}(\Sigma_{\gamma})$ respectively. Let $\psi_{\alpha} :\N \rightarrow \Malpha, \psi_{\beta}:\N \rightarrow \Mbeta$ denote the obvious maps.
	
	Let $j: E \hookrightarrow \N$ be the exceptional divisor, then $E = \psi_{\alpha}^{-1}(\phi_{\alpha}^{-1}(\Sigma_{\gamma})) = \psi_{\beta}^{-1}(\phi_{\beta}^{-1}(\Sigma_{\gamma}))$, and hence we have the fibre diagram
	
	\[
	\xymatrix{ &E \ar[ld]_{\psi_{\alpha}} \ar[rd]^{\psi_{\beta}} \\
		\phi_{\alpha}^{-1}(\Sigma_{\gamma}) \ar[rd]_{\phialpha} 
		&
		&\phi_{\beta}^{-1}(\Sigma_{\gamma}) \ar[ld]^{\phibeta} \\
		&\Sigma_{\gamma} 	
	}
	\]
	
	from which we get that $E$ is a $\Pnbeta$-bundle over $\phialpha^{-1}(\Sigma_{\gamma})$, and a $\Pnalpha$-bundle over $\phibeta^{-1}(\Sigma_{\gamma})$ via $\psi_{\alpha}$ and $\psi_{\beta}$ respectively.
	
	\begin{lemma}\label{rational}
		$\phi_{\alpha}^{-1}(\Sigma_{\gamma})$ and $\phi_{\beta}^{-1}(\Sigma_{\gamma})$ are smooth rational varieties (i.e. birational to $\mathbb{P}^n$ for some n); hence $\CH_0(\phi_{\alpha}^{-1}(\Sigma_{\gamma}))\otimes{\Q} \cong \mathbb{Q}\cong \CH_0(\phi_{\beta}^{-1}(\Sigma_{\gamma}))\otimes{\Q}$.
	\end{lemma}
	
	\begin{proof}	
		$\phi_{\alpha}^{-1}(\Sigma_{\gamma})$ is smooth since it is a projective bundle over the smooth variety $\Sigma_{\gamma}$.
		By equation (5) in \cite{BH}, $ \Sigma_{\gamma} $ is the product of two smaller dimensional moduli, which are rational (by \cite[Theorem 6.2]{BY}), so $ \Sigma_{\gamma} $ is itself rational.
		
		Since $ \phialpha^{-1}(\Sigma_{\gamma}) $ and $ \phibeta^{-1}(\Sigma_{\gamma}) $ are projective bundles over $ \Sigma_{\gamma} $, they are also rational. This proves the first assertion.
		
		Moreover, by \cite[Example 16.1.11]{Ful}, the Chow groups of 0-cycles is a birational invariant; and $\CH_0(\mathbb{P}^n) \cong \mathbb{Z} \,\forall\,n$, so we get the second assertion as well.
	\end{proof}
	Also note that $\N$ is smooth, being the blow-up of a smooth variety $\Malpha$ along a smooth subvariety $\phi_{\alpha}^{-1}(\Sigma_{\gamma})$.
	
	\begin{proposition}\label{chow1iso}
		$\psi_{\alpha}:\mathcal{N} \rightarrow \Malpha$ induces the following isomorphisms:
		\begin{eqnarray*}
			&(a)& \CH_1(\Malpha)_{hom}\otimes \Q \xrightarrow[\simeq]{\psi_{\alpha}^*} \CH_1(\N)_{hom}\otimes\Q. \\
			&(b)& H^{2n-3}(\Malpha ,\Q) \xrightarrow[\simeq]{\psi_{\alpha}^*} H^{2n-3}(\N,\Q), \,\text{and hence}\,  IJ^{n-1}(\Malpha)\xrightarrow[\simeq]{\psi_{\alpha}^*} IJ^{n-1}(\N).
		\end{eqnarray*}

	\end{proposition}
	
	\begin{proof}
		(a) By \cite[Theorem 9.27]{Voi2}, there exists an isomorphism of Chow groups		
		\begin{align}
		\CH_0(\phialpha^{-1}(\Sigma_{\gamma}))\otimes{\Q} &\bigoplus \CH_1(\Malpha)\otimes{\Q} \overset{g_{\alpha}}{\longrightarrow} \CH_1(\N)\otimes{\Q} \label{chowiso}\\
		\textnormal{given\,\,by}\quad\quad (W_0\,\,&, \,\,W_1) \longmapsto j_*((c_1(\mathcal{O}(1))^{n_{\beta}-1}\cap (\psi_{\alpha}|_E)^*(W_0)) + \psi_{\alpha}^*(W_1),
		\end{align}
		where as before $\mathcal{O}(1)$ is a very ample line bundle on $E$. Of course, a similar isomorphism $g_{\beta}$ exists for the blow-up $\psi_{\beta}: \N\rightarrow\Mbeta$ as well.
		
		Now, since $\gamma$ lies on only one hyperplane, $\Sigma_{\gamma}$ nonsingular according to \cite[section 3.1]{BH}. Hence $\phialpha^{-1}(\Sigma_{\gamma})$ and $\phibeta^{-1}(\Sigma_{\gamma})$ are also nonsingular, being projective bundles over $\Sigma_{\gamma}$.
		They are rational as well, since $\Sigma_{\gamma}$ is rational, being a product of two smaller dimensional moduli \cite[Section 3.1]{BH}.
		
		It is clear that $g_{\alpha}$ restricts to an isomorphism on the cycles homologically equivalent to zero, since the cycle class map commutes with pullback and pushforward maps on cohomology. \\
		In general, if $X$ is a nonsingular variety of dimension $n$ and $Z$ is a nonsingular closed subvariety of codimension $r$, then by \cite[Remark 11.16]{Voi1}, under the cycle class map $\CH^r(X) \rightarrow H^{2r}(X,\mathbb{Z})$ followed by Poincare duality $H^{2r}(X,\mathbb{Z})\cong H_{2n-2r}(X,\mathbb{Z})$, the image of the cycle class $[Z]$ maps to the homology class of the oriented submanifold $Z$. Moreover, since $X$ is rational, $\CH_0(X)\otimes \Q \simeq \Q \simeq H_0(X,\Q)$. Hence the composition
		\[\CH_0(X)\otimes \Q \xrightarrow{\overset{\text{cycle class}}{\text{map}}} H^{2n}(X,\Q)\overset{\overset{\text{Poincar\'e}}{\text{duality}}}{\cong} H_0(X,\Q)\]
		an isomorphism. But the elements of the subgroup $\CH_0(X)_{hom}\otimes \Q \subset \CH_0(X)\otimes \Q$ goes to zero under the composition by definition, so $\CH_0(X)_{hom}\otimes \Q = 0$. Hence we get our claim.\\
		
		(b) We also have the following blow-up formula for cohomology groups for all $k\geq 0$, given by 		
		\begin{align}
		\bigoplus_{q=0}^{n_{\beta}} H^{k-2q}(\phi_{\alpha}^{-1}(\Sigma_{\gamma}),\Q) &\oplus H^k(\Malpha ,\Q) \xrightarrow{\sim} H^k(\N,\Q), \label{cohomologyiso}\\
		(\sigma_,\cdots,\sigma_{r-1}, \,&\,\tau) \mapsto \sum_{q=0}^{n_{\beta}}j_*(c_1(\mathcal{O}_E(1))^q \cup (\psi_{\alpha}|_E)^*(\sigma)) + \psi_{\alpha}^*(\tau)
		\end{align}
		
		Put $k=2n-3$, where $n=dim \Malpha$. Note that $q\leq n_{\beta}-1 \implies 2n-3-2q \geq 2n-2n_{\beta}-1$. Since codim $\phi_{\alpha}^{-1}(\Sigma_{\gamma}) = n_{\beta}+1$, the real dimension of $\phi_{\alpha}^{-1}(\Sigma_{\gamma})$ equals $2n-2n_{\beta}-2$, hence in the LHS of \ref{cohomologyiso}, all the $H^{2n-3-2q}(\phi_{\alpha}^{-1}(\Sigma_{\gamma}),\Q)$ are zero except for $q= n_{\beta}$. But by Poincare duality,\\ $H^{2n-3-2n_{\beta}}(\phi_{\alpha}^{-1}(\Sigma_{\gamma})) \simeq H_1(\phi_{\alpha}^{-1}(\Sigma_{\gamma}))$, which is zero as well, since $\phi_{\alpha}^{-1}(\Sigma_{\gamma})$ is rational. \\
		$\therefore$ We get our claim, since $\psi_{\alpha}^*$ respects Hodge decomposition.
	\end{proof}
	
	\begin{theorem}\label{abeljacobiarbitwt}
		For any generic weight $\alpha$, $AJ^{n-1}: \CH^{n-1}(\Malpha)_{hom}\otimes \Q \rightarrow IJ^{n-1}(\Malpha)\otimes\Q$ is a split surjection. 
	\end{theorem}
	
	\begin{proof}
		The result is already proven for small weights in Corollary \ref{abeljacobismallwt}. First let $\alpha$ and $\beta$ belong to adjacent chambers separated by a single wall, and moreover assume $\alpha$ to be small. Consider the following commutative diagram:
		\begin{align*}
		\xymatrix{ \CH_1(\Malpha)_{hom}\otimes \Q \ar[r]_{\simeq}^{\psi_{\alpha}^*} \ar[d]_{AJ^{n-1}} & \CH_1(\N)_{hom}\otimes \Q \ar[d]^{AJ^{n-1}} \\
			IJ^{n-1}(\Malpha) \ar[r]_{\simeq}^{\psi_{\alpha}^*} & IJ^{n-1}(\N)
		}
		\end{align*}
		
		where the two horizontal maps are isomorphisms by Proposition \ref{chow1iso}. Since $\alpha$ is small, the left vertical Abel-Jacobi map is a split surjection by Corollary \ref{abeljacobismallwt}, from which we get that the right vertical arrow is also a split surjection. Replacing $\alpha$ by $\beta$ in the diagram above, we conclude that the Abel-Jacobi map is a split surjection for $\Mbeta$ as well.\\
		Recall that there are only finitely many walls in the set $V_m$ of all admissible weights, and the moduli spaces corresponding to weights in the same chamber are isomorphic. If $\alpha$ is any arbitrary generic weight, we can choose a small weight $\beta$ and a sequence of generic weights which starts at $\alpha$ and ends at $\beta$, such that each successive weights are separated by a single wall; in that case, by the remark above, since the statement holds for $\beta$, it holds for $\alpha$ as well.
	\end{proof}

	\section{Abel-Jacobi isomorphism for codimension 2 cycles in the rank 2 case}
	We keep the same notation as in Section 3, with the only modification being that we fix the rank $r=2$ here.
	
	\subsection{CASE OF SMALL WEIGHTS} Let $\alpha$ be small, as in section \ref{smallweight}. In this case, the canonical morphism $\pi: \Malpha \rightarrow \M$, make $\Malpha$ into a $(\mathbb{P}^1)^m$-bundle over $\mathcal{M}$, where $m$ equals number of Parabolic points, by \cite[Theorem 4.2]{BY}. 	
	\begin{lemma}\label{chowiso1}
		(a) $\CH^1(\Malpha)_{hom}\otimes \Q = 0.$\\
		\hspace*{11.0ex} (b) $\pi^*$ induces isomorpism $\CH^2(\M)_{hom}\otimes \Q \simeq \CH^2(\Malpha)_{hom}\otimes \Q.$
	\end{lemma}
	
	\begin{proof}
		(a) First, let there be only one Parabolic point. In that case, $\pi: \Malpha\rightarrow \M$ is a $\mathbb{P}^1$-bundle, since $\alpha$ is small. Therefore, by the projective bundle formula for Chow groups (\cite[Theorem 9.25]{Voi2}), putting $n = dim\Malpha$,
		\begin{align*}
		&\CH_{n-1}(\Malpha) \simeq \CH_{n-2}(\M) \oplus \CH_{n-1}(\M) \\
		\implies & \CH^1(\Malpha) \simeq \CH^1(\M) \oplus \CH^0(\M) \hspace{10ex} (\because dim \M = n-1)\\
		\implies &\CH^1(\Malpha)_{hom} \simeq \CH^1(\M)_{hom}\oplus \CH^0(\M)_{hom}
		\end{align*}
		
		But clearly $\CH^0(\M,\Q) = \Q,$ and $\CH^1(\M,\Q) = \Q$ by \cite[Th\'{e}oreme B]{DN}, so \\
		$\CH^0(\M)_{hom}\otimes \Q=\CH^1(\M)_{hom}\otimes \Q =0$, and hence 
		\begin{align*}
		\CH^1(\Malpha)_{hom}\otimes\Q = 0.
		\end{align*}
		
		In general, $\Malpha$ has the following descrpition: if there are  $m$ distinct Parabolic points $S=\{p_1,\cdots, p_m\}$, for each $1\leq i\leq m$ let $ \mathcal{E}_i $ denote the restriction of the universal bundle over $ X\times \mathcal{M}$ to $ \{p_i\}\times \mathcal{M} $. Then an analogous argument as above shows that $\mathcal{M}_{\alpha}$ is isomorphic to the fiber product of $\mathbb{P}(\mathcal{E}_i)$'s over $\mathcal{M}$, i.e. \begin{align*}
		\mathcal{M}_\alpha \cong \mathbb{P}(\mathcal{E}_1)\times_{\mathcal{M}} \mathbb{P}(\mathcal{E}_2)\times_{\mathcal{M}} ...\times_{\mathcal{M}}\mathbb{P}(\mathcal{E}_m).
		\end{align*}
		
		For each $1\leq i\leq m$, let $ \mathcal{F}_i:= \mathbb{P}(\mathcal{E}_1)\times_{\mathcal{M}} \mathbb{P}(\mathcal{E}_2)\times_{\mathcal{M}} ...\times_{\mathcal{M}} \mathbb{P}(\mathcal{E}_i) $.\newline
		Here $\Malpha \cong \mathcal{F}_m$, so we have the following fiber diagram:
		\begin{align}\label{diagram}
		\xymatrix{\Malpha \ar[r] \ar[d] 
			& \mathbb{P}(\mathcal{E}_m) \ar[d] \\
			\mathcal{F}_{m-1} \ar[r] 
			& \mathcal{M}
		}
		\end{align}
		
		By induction on $m$, we have $\CH^1(\mathcal{F}_{m-1})\otimes\Q = 0$; moreover, the left vertical arrow above is a $\mathbb{P}^1$-bundle, and so by the same argument as above, 
		\begin{align}
		\CH^1(\Malpha)_{hom}\otimes \Q \simeq \CH^1(\mathcal{F}_{m-1})_{hom}\otimes \Q =0.
		\end{align}
		
		(b) Again, first let there be only one Parabolic point $p$; by the projective bundle formula for Chow groups (\cite[Theorem 9.25]{Voi2}),
		\begin{align*}
		&\CH_{n-2}(\Malpha) \simeq \CH_{n-3}(\M) \oplus \CH_{n-2}(\M) \hspace{15ex}(n= dim \Malpha) \nonumber\\
		\implies & \CH^2(\Malpha) \simeq \CH^2(\M) \oplus \CH^1(\M) \hspace{17ex} (\because dim \M = n-1)\\
		\implies &\CH^2(\Malpha)_{hom} \simeq \CH^2(\M)_{hom}\oplus \CH^1(\M)_{hom}
		\end{align*}
		Again, since $\CH^1(\M,\Q) = \Q$, we get
		\begin{align*}
		\CH^2(\Malpha)_{hom}\otimes \Q \simeq \CH^2(\M)_{hom}\otimes \Q.
		\end{align*}
		
		In general, for $m$ Parabolic points, by the left vertical arrow in \ref{diagram}, which is a $\mathbb{P}^1$-bundle, we would get 
		\begin{align*}
		\CH^2(\Malpha)_{hom}\otimes\Q \simeq \CH^2(\mathcal{F}_{m-1})_{hom}\otimes \Q \,\oplus\, \CH^1(\mathcal{F}_{m-1})_{hom}\otimes \Q
		\end{align*} 
		By (a), we have $\CH^1(\mathcal{F}_{m-1})_{hom}\otimes \Q = 0$, and by induction on $m$, we have $\CH^2(\mathcal{F}_{m-1})_{hom}\otimes \Q \simeq \CH^2(\M)_{hom}\otimes \Q$, and hence we get our result.
	\end{proof}
	
	\begin{proposition}\label{ajsmallwt}
		The Abel-Jacobi map $AJ^2: \CH^2(\Malpha)_{hom}\otimes\Q \rightarrow IJ^2(\Malpha)\otimes\Q$ is an isomorphism.
	\end{proposition}
	\begin{proof}
		By Proposition \ref{specseq}, $\pi^* : H^3(\M,\Q) \rightarrow H^3(\Malpha,\Q)$ is an isomorphism. Of course, the isomorphism is also valid if we consider cohomology with $\mathbb{C}$-coefficients, which is also an isomorphism of Hodge structures; hence going modulo the appropriate subgroups on both sides, we get an induced isomorphism on the intermediate Jacobians:
		\begin{align*}
		\pi^*: IJ^2(\M)\otimes \Q \simeq IJ^2(\Malpha)\otimes \Q.
		\end{align*}
		Now, consider the commutative diagram:	
		\begin{align}
		\xymatrix{ \CH^2(\M)_{hom}\otimes \Q \ar[r]_{\simeq}^{\pi^*} \ar[d]_{AJ^2} & \CH^2(\Malpha)_{hom}\otimes \Q \ar[d]^{AJ^2} \\
			IJ^2(\M)\otimes\Q \ar[r]_{\simeq}^{\pi^*} & IJ^2(\Malpha)\otimes\Q
		}
		\end{align}
		By Lemma \ref{chowiso1} (b) the top horizontal map is an isomorphism. The left vertical map is an isomorphism as well by \cite[Lemma 3.10 (a)]{JY} together with \cite[Proposition 3.11]{JY}, noting that even though the Proposition in \cite{JY} is only proved in the case when determinant $\mathcal{L}\simeq \mathcal{O}(-x)$ for a point $x\in C$, the argument goes through for any line bundle of degree 1, since the proof only relies on the fact that $\mathcal{M}$ is rational. Therefore we conclude that the right vertical map is an isomorphism as well.	
	\end{proof} 
	
	\subsection{CASE OF ARBITRARY WEIGHTS}
	Recall the discussion in the beginning of Section \ref{arbitraryweight}. Let $\alpha$ and $\beta$ be two generic weights lying in adjacent chambers in $V_m$ separated by a single wall. Recall that $\N$ is a common blow-up over $\Malpha$ and $\Mbeta$ along $\phi_{\alpha}^{-1}(\Sigma_{\gamma})$ and $\phi_{\beta}^{-1}(\Sigma_{\gamma})$ respectively. 
	Let $\psi_{\alpha}, \psi_{\beta}$ denote the usual projections from $\N$ to $\Malpha, \Mbeta$ respectively.	
	
	We work with $\psi_{\alpha}: \N \rightarrow \Malpha$, the case of $\psi_{\beta}$ being identical. 
	
	\begin{lemma}
		$\psi_{\alpha}^*$ induces an isomorphism $IJ^2(\N)\otimes \Q \simeq IJ^2(\Malpha)\otimes\Q.$
	\end{lemma}
	\begin{proof}
		By the blow-up formula for Cohomology, we get
		\begin{align}
		H^3(\N,\Q) \simeq H^3(\Malpha,\Q) \oplus H^1(\phialpha^{-1}(\Sigma_{\gamma}),\Q)
		\end{align}
		
		Also, by \ref{rational}, $\phialpha^{-1}(\Sigma_{\gamma})$ is rational, hence $H^1(\phialpha^{-1}(\Sigma_{\gamma}),\Q) =0 $.
		Moreover, $\psi_{\alpha}^*$ is a map of Hodge structures, so we get our claim.
	\end{proof}
	
	\begin{proposition}\label{sigmagammadescription}
		Let $\mathcal{M}(1,d,\sigma)$ denote the moduli space of Parabolic line bundles of degree $d$ and weight $\sigma$. Then $\Sigma_{\gamma}\simeq \coprod_{i=1}^k \mathcal{M}(1,d_i,\alpha^{i})$ for some non-negative integers $d_i$ weights $\alpha^i$. Here, the integers $d_i$ are not necessarily distinct. 
	\end{proposition}
	
	\begin{proof}
		Let $[E_*]\in\Sigma_{\gamma}$ be the S-equivalence class of a $\gamma$-semistable bundle $E_*$. Since $E_*$ is not $\gamma$-stable and $rank(E) =2$, there exists a line sub-bundle $L$ of $E$ such that with the induced Parabolic structure, the Parabolic sub-bundle $L_*$ has same Parabolic slope as $E_*$. \\
		Now, $\mathcal{L} = det(E) \simeq L\otimes (E/L) \implies (E/L)\simeq L^{-1}\otimes \mathcal{L}$. If $(L^{-1}\otimes \mathcal{L})_*$ denotes the induced quotient Parabolic structure from $E_*$, then consider the short exact sequence of Parabolic bundles 
		\[0\rightarrow L_*\rightarrow E_* \rightarrow (L^{-1}\otimes \mathcal{L})_*\rightarrow 0\]
		If $\mu := Par\mu(L_*) = Par\mu(E_*)$ and $\mu':= Par\mu(L^{-1}\otimes \mathcal{L})_*$, then $\mu = \dfrac{\mu+\mu'}{2} \implies \mu =\mu'$. Recalling the discussion from Section \ref{sub-2.5}, we conclude that
		\[gr_{\gamma}(E_*) = L_*\oplus (L^{-1}\otimes \mathcal{L})_* \implies [E_*] = [L_*\oplus (L^{-1}\otimes \mathcal{L})_*]. \]
		Let $d:= deg(L)$. Then $deg(L^{-1}\otimes \mathcal{L}) = -d-1$, and $d\geq 0$ (resp. $d<0$) $\iff -d-1<0$ (resp. $-d-1\geq 0$), so either $L$ or $L^{-1}\otimes \mathcal{L}$ has non-negative degree, but not both. Clearly, the condition $Par\mu(L_*) = Par\mu(E_*)$ only allows for only finitely many possible choices for $deg(L)$. 
		Therefore, we get a map 
		\begin{align*}
		\Sigma_{\gamma} &\overset{\varphi}{\longrightarrow} \coprod_{i=1}^{k}\mathcal{M}(1,d_i,\alpha^{i})\\
		[E_*] &\longmapsto \begin{cases}
		[L_*],\,\,if\,\,d=deg(L)\geq 0 \\
		[(L^{-1}\otimes \mathcal{L})_*], \,\,if\,\,d<0
		\end{cases}
		\end{align*}
		
		where $\alpha^{i}$'s are the weights induced from $\alpha$, and $d_i$'s belong to the set of all possible non-negative degrees for $deg(L)$. We note that it is possible for the same line bundle to get different Parabolic structures by being sub-bundles of different Parabolic bundles in $\Sigma_{\gamma}$, so the $d_i$'s may not be distinct, but $\alpha^{i}$'s will be necessarily distinct. \\
		To see that $\varphi$ is surjective, let $[L_*]\in \mathcal{M}(1,d_i,\alpha^{i})$ for some $i$. We define a Parabolic structure on $\mathcal{M}(1,d_i,\alpha^{i})$ as follows: at each Parabolic point $P$, we have the chain $0\leq \alpha_{P,1}<\alpha_{P,2}<1$, and $\alpha^i_P\in\{\alpha_{P,1},\alpha_{P,2}\}$. Say $\alpha^i_P = \alpha_{P,1}$. Then we define a Parabolic structure on $L^{-1}\otimes \mathcal{L}$ by choosing $\alpha_{P,2}$ at $P$. We do this for every Parbolic point to get $(L^{-1}\otimes \mathcal{L})_*$. Note that $d_i$'s are chosen such that the condition $Par\mu(L_*)= Par\mu(E_*) = Par\mu((L^{-1}\otimes \mathcal{L})_*$ is satisfied, and by the choice of $d_i$'s,  $Par\mu(L_*) = Par\mu((L^{-1}\otimes \mathcal{L})_*)$ is automatically satisfied, so that $E_* :=L_*\oplus (L^{-1}\otimes \mathcal{L})_*$ belongs to $\Sigma_{\gamma}$, and maps to $[L_*]$.
		
		To see $\varphi$ is injective, let $\varphi([E_*]) = \varphi([E'_*])$, suppose $gr_{\gamma}(E_*) = L_*\oplus (L^{-1}\otimes \mathcal{L})_*$, $gr_{\gamma}(E'_*) = L'_*\oplus (L'^{-1}\otimes \mathcal{L})_* $, where $L\subset E, L'\subset E'$ are line sub-bundles, and let $d:=deg(E),d':=deg(E')$. Without loss of generality assume that $d\geq 0$; if $d'\geq 0$ as well, by the definition of $\varphi$ we get $L_*\simeq L'_*$. But notice that an isomorphism of Parabolic line bundles forces the weights of $L_*$ and $L'_*$ at each Parabolic points to be equal; otherwise, by definition of a morphism of Parabolic bundles (see Defintion \ref{defparhomo}), if the weight of of $L_*$ at $P$ is bigger than the weight of $L'_*$ at $P$, then the morphism restricted to the fibers at $P$ must be zero, contradicting that it is an isomorphism. This also forces $L\simeq L'$ as usual line bundles, which induces an isomorphism $L^{-1}\otimes \mathcal{L}\simeq L'^{-1}\otimes \mathcal{L}$, and since the Parabolic weights of $L_*$ and $L'_*$ agree at each $P$, the same must be true for the induced Parabolic structure on $L^{-1}\otimes \mathcal{L}$ and $L'^{-1}\otimes \mathcal{L}$ as well. Hence we get $(L^{-1}\otimes \mathcal{L})_*\simeq (L'^{-1}\otimes \mathcal{L})_*$. Finally, combining the two isomorphisms, we get $L_* \oplus (L^{-1}\otimes \mathcal{L})_* \simeq L'_*\oplus (L'^{-1}\otimes \mathcal{L})_* \implies [E_*]=[E'_*]$.\\
		If $d'<0$, we will get $L_*\simeq (L^{-1}\otimes \mathcal{L})_*$, and the same argument above goes through.\\		
		Therefore we get our claim.
	\end{proof}
	
	\begin{lemma}
		Let $n=dim\N = dim\Malpha$. $\psi_{\alpha}:\N \rightarrow \Malpha$ induces isomorphism \[\psi_{\alpha}^*: \CH_{n-2}(\Malpha)_{hom}\otimes\Q \simeq \CH_{n-2}(\N)_{hom}\otimes\Q\]
	\end{lemma}
	
	\begin{proof}
		We observe that since codim$\Sigma_{\gamma} = 1+ n_{\alpha}+n_{\beta}$ and $\phi_{\alpha}^{-1}(\Sigma_{\gamma})$ is a $\mathbb{P}^{n_{\alpha}}$-bundle over $\Sigma_{\gamma}$, we get codim$(\phi_{\alpha}^{-1}(\Sigma_{\gamma})) = 1+ n_{\beta}$, hence by the blow-up formula for Chow groups (\cite[Theorem 9.27]{Voi2}), 
		
		\begin{align}
		\CH_{n-2}(\N) &\simeq \CH_{n-2}(\Malpha)\oplus \underset{0\leq k\leq n_{\beta}-1}{\bigoplus} \CH_{n-2-n_{\beta}+k}(\phi_{\alpha}^{-1}(\Sigma_{\gamma}))\nonumber \\
		&= \CH_{n-2}(\Malpha)\oplus \CH_{n-n_{\beta}-2}(\phialpha^{-1}(\Sigma_{\gamma}))\oplus \CH_{n-n_{\beta}-1}(\phialpha^{-1}(\Sigma_{\gamma})) \nonumber\\
		\implies \CH^2(\N) & \simeq \CH^2(\Malpha)\oplus \CH^1(\phialpha^{-1}(\Sigma_{\gamma})) \oplus \CH^0(\phialpha^{-1}(\Sigma_{\gamma}))\label{eqn}
		\end{align}
		Since $\CH^0(\phi_{\alpha}^{-1}(\Sigma_{\gamma})) = \mathbb{Z}$, clearly $\CH^0(\phialpha^{-1}(\Sigma_{\gamma}))_{hom}\otimes\Q=0$. We need to show that $\CH^1(\phialpha^{-1}(\Sigma_{\gamma}))_{hom}\otimes \Q=0$ as well. Recall that $\phialpha^{-1}(\Sigma_{\gamma})$ is a $\mathbb{P}^{n_{\alpha}}$-bundle over $\Sigma_{\gamma}$, and\\ dim$(\phialpha^{-1}(\Sigma_{\gamma}))=n-n_{\beta}-1$. Using the projective bundle formula for Chow groups (\cite[Theorem 9.25]{Voi2}) we get
		\begin{align}
		&\CH^1(\phialpha^{-1}(\Sigma_{\gamma})) \simeq \CH^1(\Sigma_{\gamma}) \oplus \CH^0(\Sigma_{\gamma})\nonumber\\
		\implies  &\CH^1(\phialpha^{-1}(\Sigma_{\gamma}))_{hom}\otimes\Q \simeq \CH^1(\Sigma_{\gamma})_{hom}\otimes\Q \,\,\oplus\,\, \CH^0(\Sigma_{\gamma})_{hom}\otimes \Q \nonumber
		\end{align}
		Since $\CH^0(\Sigma_{\gamma})=\mathbb{Z}$, clearly $\CH^0(\Sigma_{\gamma})_{hom}\otimes\Q=0$. Moreover, by Proposition \ref{sigmagammadescription}, $\CH^1(\Sigma_{\gamma}) =\CH^1(\coprod_{i=1}^{k} Pic^{d_i}(C)) =\oplus_{i=1}^{k}\CH^1(Pic^{d_i}(C))$, and since $\CH^1(Pic^{d_i}(C)) = \mathbb{Z}$, we get $\CH^1(Pic^{d_i}(C))_{hom}\otimes \Q = 0\, \forall i$, and hence $\CH^1(\Sigma_{\gamma})_{hom}\otimes\Q=0$ as well. 
		So we conclude from \ref{eqn}.

	\end{proof}
	
	\begin{theorem}\label{abeljacobiiso}
		For any generic weight $\alpha$, the Abel-Jacobi map $AJ^2: \CH_{n-2}(\Malpha)_{hom}\otimes\Q \rightarrow IJ^2(\Malpha)\otimes\Q$ is an isomorphism.
	\end{theorem}
	
	\begin{proof}
		Follows from essentially the same argument as in Theorem \ref{abeljacobiarbitwt}, where instead of the commutative diagram used in the proof, we use the following diagram instead, valid for all weights $\alpha$:
		\begin{align*}
		\xymatrix{ \CH_{n-2}(\Malpha)_{hom}\otimes\Q \ar[r]^{\simeq} \ar[d]^{AJ^2} & \CH_{n-2}(\N)_{hom}\otimes \Q \ar[d]_{AJ^2} \\
			IJ^2(\Malpha)\otimes \Q \ar[r]^{\simeq} & IJ^2(\N)\otimes\Q
		}
		\end{align*}
	\end{proof}	
	

\begin{thebibliography}{AAAA}
		\bibitem[JY]{JY} J. N. Iyer and J. Lewis, The Abel-Jacobi isomorphism for one cycles on Kirwan's resolution of the moduli space $\mathcal{SU}_C(2,\mathcal{O}_C)$, \textit{Journal F\"{u}r Die Reine Und Angewandte Mathematik}, Volume 2014(696), 1--29
		\bibitem[BH]{BH} H. Boden and Y. Hu, Variations of Moduli of Parabolic Bundles, \textit{Mathematische Annalen}, Volume 301 (1995), 539--559
		\bibitem[BY]{BY} H. Boden and K. Yokogawa, Rationality of the moduli space of Parabolic bundles, \textit{Journal of the London Mathematical Society}, Volume 59 (1999), 461--478
		\bibitem[Ful]{Ful} W. Fulton, Intersection Theory, \textit{Springer-Verlag Berlin Heidelberg} (1998)
		\bibitem[KS]{KS} A. King and A. Schofield, Rationality of moduli of vector bundles on curves, \textit{Indagationes Mathematicae}
		Volume 10, Issue 4 (1999), 519--535
		\bibitem[MS]{MS} V. B. Mehta and C.S. Seshadri, Moduli of Vector Bundles on Curves with Parabolic Structures, \textit{Mathematische Annalen} Volume 248 (1980), 205 -- 240
		\bibitem[Voi1]{Voi1} C. Voisin, Hodge Theory and Complex Algebraic Geometry Vol. I, \textit{Cambridge Studies in Advanced Mathematics}
		\bibitem[Voi2]{Voi2} C. Voisin, Hodge Theory and Complex Algebraic Geometry Vol. II, \textit{Cambridge Studies in Advanced Mathematics}
		\bibitem[EV]{EV} H. Esnault and E. Viewhag, Deligne-Beilinson cohomology, Beilinson's Conjectures on Special Values of L-Functions, \textit{Perspect. Math.}, 4, Academic press, Boston, MA, (1988), 43--91
		\bibitem[BL]{BL} C. Birkenhake and H. Lange, Complex Abelian Varieties, \textit{Springer-Verlag Berlin Heidelberg}
		\bibitem[AS]{AS} D. Arapura and P. Sastry, Intermediate Jacobians and Hodge Structures of Moduli Spaces, \textit{Proceedings Mathematical Sciences} volume 110 (2000), 1--26
		\bibitem[DN]{DN} J. Drezet, M.S. Narasimhan, Groupe de Picard des vari\'{e}t\'{e}s de modules de fibr\'{e}s semi-stables sur les courbes alg\'{e}briques, \textit{Inventiones Mathematicae} 97 (1989), 53--94 
	\end{thebibliography}
\end{document}